\documentclass{amsart}
\usepackage{amsmath,amssymb,hyperref,mathrsfs,color}
\usepackage{paralist}
\usepackage{calc}
 \newtheorem{The}{Theorem}[section]
 \newtheorem{Cor}[The]{Corollary}
 \newtheorem{Lem}[The]{Lemma}
 \newtheorem{Pro}[The]{Proposition}
 \theoremstyle{definition}
 \newtheorem{defn}[The]{Definition}
 \theoremstyle{remark}
 \newtheorem{Rem}[The]{Remark}
 \newtheorem*{ex}{Example}
 \numberwithin{equation}{section}
\newcommand{\T}{\mathbb{T}}
\newcommand{\R}{\mathbb{R}}
\newcommand{\Z}{\mathbb{Z}}
\newcommand{\N}{\mathbb{N}}

\title[propagation of singularities]{Propagation of singularities for weak KAM solutions and barrier functions}
\author{Piermarco Cannarsa \and Wei Cheng \and Qi Zhang}
\address{Dipartimento di Matematica, Universit\`a di Roma Tor Vergata,
Via della Ricerca Scientifica 1, 00133 Roma, Italy}
\email{cannarsa@mat.uniroma2.it}
\address{Department of Mathematics, Nanjing University,
Nanjing 210093, China}
\email{chengwei@nju.edu.cn}
\thanks{This work was partially supported by the Natural Scientific Foundation of China (Grant No. 11271182) and National Basic Research Program of China (Grant No. 2013CB834100)}
\address{Department of Mathematics, Nanjing University of Aeronautics and Astronautics, Nanjing 210016, China}
\email{zhangqi@nuaa.edu.cn}
\date{\today}
\subjclass[2010]{26B25, 35A21, 49L25, 37J50, 70H20}
\keywords{Semiconcave functions, singularities, Hamilton-Jacobi equations,  weak KAM theory.}
\begin{document}
\maketitle

\begin{abstract}
This paper studies the structure of the singular set (points of nondifferentiability) of viscosity solutions to Hamilton-Jacobi equations associated with general mechanical systems on the $n$-torus. First, using the level set method, we characterize the propagation of singularities along  generalized characteristics. Then, we obtain a local  propagation result for singularities of weak KAM solutions in the supercritical case. Finally, we  apply such a result to study the propagation of singularities for barrier functions.
\end{abstract}

\section{Introduction}

Let $\T^n$ be the $n$-dimensional flat torus, and let $H(x,p)$ be a Tonelli Hamitonian on $T^{\ast}\T^n\backsimeq\T^n\times\R^n$, that is, a continuous function  in $(x,p)$,   convex and superlinear with respect to  $p$. For any $c\in\R^n$, we denote by $\alpha(c)$  the unique real number such that the associated Hamilton-Jacobi equation
\begin{equation}\label{HJE_intro}
H\big(x,c+Du(x)\big)=\alpha(c),\quad x\in\T^n,
\end{equation}
admits a $\T^n$-periodic viscosity solution, $u_c$. The map $c\mapsto\alpha(c)$ is called Mather's $\alpha$-function (or {\em effective Hamiltonian}) in the literature.

Viscosity solutions (or weak KAM solutions) of \eqref{HJE_intro} have been widely studied for  their relevance to many different research fields such as nonlinear partial differential equations, calculus of variations,  optimal control, optimal transport, geodesic dynamics, and Hamiltonian dynamical systems (see, for instance, \cite{li82, lipava88}, \cite{ev92}, \cite{Fathi-book,Fathi-Siconolfi2004}, \cite{Cannarsa-Sinestrari}).
 As is well-known, the  solutions of \eqref{HJE_intro} turn out to be semiconcave with linear modulus and coincide with the value function of the action functional defined by the associated Tonelli Lagrangian. Semiconcavity is a very useful property for the analysis of this paper: we recall the basic notions about  semiconcave functions in section~\ref{sse:semiconcave}, referring the reader  to the monograph \cite{Cannarsa-Sinestrari} for more  details and  applications.
In particular, the superdifferential $D^+v$ of a semiconcave function $v$ is a key notion in this theory, as is the singular set of $v$, $\Sigma_v$, which is defined as the set of all points  $x$ such that $D^+v(x)$ is not a singleton.

In this paper, we will be concerned with some questions related to Mather's theory and weak KAM theory for Hamiltonian dynamical systems, especially in connection with the dynamics of singularities on a supercritical energy surface, i.e., when $\alpha(c)$ in \eqref{HJE_intro} is strictly greater than Ma\~n\'e's critical value. We will focus our attention on the class of Hamiltonians which is relevant to mechanics, that is,
\begin{equation}\label{eq:hamilton}
H(x,p)=\frac 12\langle A(x)p,p\rangle+V(x),\quad (x,p)\in\T^n\times\R^n,
\end{equation}
where $A(x)$ is a  symmetric  positive definite  $n\times n$ matrix, $\Z^n$-periodic with respect to $x\in\R^n$, and $V$ is a smooth $\Z^n$-periodic  function on $\R^n$.

Given a  weak KAM solution $u_c$ of  \eqref{HJE_intro}  (here we regard $u_c$ as a $\Z^n$-periodic function on $\R^n$), one is interested in describing the structure of the singularities of $u_c$ or, equivalently, of the function
\begin{equation*}
v(x):=u_c(x)+\langle c,x\rangle ,\quad x\in\R^n,
\end{equation*}
in a neighborhood of a singular point $x_0\in\Sigma_v$.
We recall that $\Sigma_v$ has Lebesgue measure zero and  is  countably $(n-1)$-rectifiable (\cite{Za}).

For applications to dynamical systems, the above measure theoretic results need to be completed by a suitable  analysis of the way singularities propagate. An interesting approach to such a problem---developed by various authors in several papers~(\cite{Dafermos}, \cite{Ambrosio-Cannarsa-Soner,Albano-Cannarsa,Cannarsa-Yu,acns}, \cite{Yu})---is the one based on the  differential inclusion
\begin{equation}\label{eq:char}
\dot{\mathbf{x}}(s)\in\mathrm{co}\, H_p\big(\mathbf{x}(s),D^+v(\mathbf{x}(s))\big),
\end{equation}
where ``co'' stands for the closed convex hull. The above  inclusion generalizes the classical equation of characteristics and turns out to be very useful to describe singular dynamics. Indeed, given $x_0\in\Sigma_v$, the solution of \eqref{eq:char} with initial condition  $\mathbf{x}(0)=x_0$ provides  a singular arc for $v$ under the nondegeneracy condition
\begin{equation}\label{eq:nondeg}
0\not\in D^+v(x_0)
\end{equation}
(see \cite{Albano-Cannarsa,Cannarsa-Yu} and \cite{Yu}). This condition has an interesting geometrical meaning, as it is equivalent to the fact that the level set
\begin{equation*}
\Lambda_{x_0}:=\big\{x\in\R^n~: ~v(x)=v(x_0)\big\}
\end{equation*}
is an $(n-1)$-dimensional Lipschitz submanifold near $x_0$.

On the other hand,  when \eqref{eq:nondeg} is violated, the unique of solution of \eqref{eq:char} starting from $x_0$ is the constant arc ${\mathbf{x}}(s)\equiv x_0$. So, new ideas are needed to describe the structure of $\Sigma_v$ near a  singular point at which  $D^+v$ contains the zero vector.  The main results of this paper propose possible solutions to  such a problem in the supercritical case
\begin{equation*}
\alpha(c)>\max_{\T^n} V
\end{equation*}
that was partially addressed in \cite{Cannarsa-Yu}.

First, we  show that condition \eqref{eq:nondeg}  holds  if and only if $x_0$ is not a local maximum point of $v$ and the level set $\Lambda_{x_0}$ is an $(n-1)$-dimensional Lipschitz submanifold near $x_0$ (Theorem \ref{equiv}). In any case, that is, even when  $0\in D^+v(x_0)$, we  prove that  the singularities of $v$ propagate along Lipschitz arcs (Theorem \ref{th:wKAM}),
solving an open problem raised in \cite{Cannarsa-Yu}. For this result, we use the  propagation  principle for general semiconcave functions due to Albano and the first author~\cite{Albano-Cannarsa} (see also \cite{Cannarsa-Yu}).

Finally, we turn to study the singularities of the so-called barrier function $B^{\ast}_c$ (see Mather~\cite{Mather93}), concluding that  propagation also occurs in this case (Theorem~\ref{pro_barrier}). Moreover,  given a pair of weak KAM solutions $(u^-_c,u^+_c)$ such that $B^{\ast}_c=u^-_c-u^+_c$, we show that any $x\in\Sigma_{B^{\ast}_c}$ produces a homoclinic orbit with respect to the Aubry set provided that $u^-_c$ and $u^+_c$ possess a common reachable gradient
(Theorem \ref{homoclinic_orbit}). In a forthcoming paper,  building on this analysis, we will address  further regularity issues aiming at solving  central problems in Hamiltonian dynamics such as the problem of Arnold diffusions.

The paper is organized as follows. In Section~2, we review  preliminary material on  viscosity solutions of  Hamilton-Jacobi equations and  semiconcavity. In Section~3, we  characterize the regularity condition $0\not\in D^+v(x)$, and discuss the (local and global) propagation of singularities along  generalized characteristics. In section~4,  the local propagation of singularities is derived for the barrier function and applied to the study of homoclinic orbits.

\section{preliminaries}
Let $\T^n$ be the $n$-dimensional flat torus. We denote by $T\T^n$ the tangent bundle of $\T^n$ and by $T^{\ast}\T^n$ the cotangent bundle.

\begin{defn}
A function $L:T\T^n\to\R$ is said to be a {\em Tonelli Lagrangian} if the following assumptions are satisfied.
\begin{enumerate}[(L1)]
  \item {\em Smoothness}: $L=L(x,q)$ is of class at least $C^2$.
  \item {\em Convexity}: The Hessian $\frac{\partial^2 L}{\partial q^2}(x,q)$ is positive definite on each fibre $T_x\T^n$.
  \item {\em Superlinearity}:
$$
\lim_{|q|\to\infty}\frac{L(x,q)}{|q|}=\infty\quad\text{uniformly for } x\in\T^n.
$$
\end{enumerate}

\end{defn}

Given a Tonelli Lagrangian  $L$,  the {\em Tonelli Hamiltonian} $H=H(x,p)$ associated with  $L$ is defined as follows:
$$
H(x,p)=\max\big\{\langle p,q\rangle-L(x,q): q\in T_x\T^n\big\},\quad (x,p)\in T^{\ast}\T^n\,.
$$
It is easy to see that for any Tonelli Lagrangian $L$, the associated Hamiltonian $H$ satisfies  similar smoothness ($H$ is of class at least $C^2$), convexity, and superlinearity conditions, which will be referred to as (H1), (H2), and (H3).

\subsection{Semiconcave functions}
\label{sse:semiconcave}
Let $\Omega\subset\R^n$ be a convex open set. We recall that  a function $u:\Omega\to\R$ is {\em semiconcave} (with linear modulus) if there exists a constant $C>0$ such that
\begin{equation}\label{eq:SCC}
\lambda u(x)+(1-\lambda)u(y)-u(\lambda x+(1-\lambda)y)\leqslant\frac C2\lambda(1-\lambda)|x-y|^2
\end{equation}
for any $x,y\in\Omega$ and $\lambda\in[0,1]$.  Any constant $C$ that satisfies the above inequality  is called a {\em semiconcavity constant} for $u$ in $\Omega$. Property \eqref{eq:SCC} has many equivalent versions, one of which is to require that
\begin{equation}\label{eq:SCC2}
x\mapsto u(x)-\frac C2|x|^2\;\;\text{is concave in}\;\;\Omega\,.
\end{equation}

 Let $u:\Omega\subset\R^n\to\R$ be a continuous function. We recall that, for any $x\in\Omega$, the closed convex sets
\begin{align*}
D^-u(x)&=\left\{p\in\R^n:\liminf_{y\to x}\frac{u(y)-u(x)-\langle p,y-x\rangle}{|y-x|}\geqslant 0\right\},\\
D^+u(x)&=\left\{p\in\R^n:\limsup_{y\to x}\frac{u(y)-u(x)-\langle p,y-x\rangle}{|y-x|}\leqslant 0\right\}.
\end{align*}
are called the {\em (Dini) subdifferential} and {\em superdifferential} of $u$ at $x$, respectively.
For semiconcave functions, the superdifferential plays a major role as is shown by
the following property (see, e.g., \cite{Cannarsa-Sinestrari} for the proof).
\begin{Pro}
\label{criterion-Du_semiconcave}
Let $u:\Omega\to\R$ be a continuous function. If there exists a constant $C>0$ such that, for any $x\in\Omega$, there exists $p\in\R^n$ such that
\begin{equation}\label{criterion_for_lin_semiconcave}
u(y)\leqslant u(x)+\langle p,y-x\rangle+\frac C2|y-x|^2,\quad \forall y\in\Omega,
\end{equation}
then $u$ is semiconcave with constant $C$ and $p\in D^+u(x)$.
Conversely,
if $u$ is semiconcave  in $\Omega$ with constant $C$, then \eqref{criterion_for_lin_semiconcave} holds for any $x\in\Omega$ and $p\in D^+u(x)$.
\end{Pro}
Let $u:\Omega\to\R$ be locally Lipschitz. We recall that a vector $p\in\R^n$ is called a {\em reachable} (or {\em limiting}) {\em gradient}  of $u$ at $x$ if there exists a sequence $\{x_n\}\subset\Omega\setminus\{x\}$ such that $u$ is differentiable at $x_k$ for each $k\in\N$, and
$$
\lim_{k\to\infty}x_k=x\quad\text{and}\quad \lim_{k\to\infty}Du(x_k)=p.
$$
The set of all reachable gradients of $u$ at $x$ is denoted by $D^{\ast}u(x)$.

Now we list some   well known properties of the superdifferential  of a semiconcave function on $\Omega\subset\R^n$ (see, e.g., \cite{Cannarsa-Sinestrari} for the proof).

\begin{Pro}\label{basic_facts_of_superdifferential}
Let $u:\Omega\subset\R^n\to\R$ be a semiconcave function and let $x\in\Omega$. Then the following properties hold.
\begin{enumerate}[\rm {(}a{)}]
  \item $D^+u(x)$ is a nonempty compact convex set in $\R^n$ and $D^{\ast}u(x)\subset\partial D^+u(x)$, where  $\partial D^+u(x)$ denotes the topological boundary of $D^+u(x)$.
  \item The set-valued function $x\rightsquigarrow D^+u(x)$ is upper semicontinuous.
  \item If $D^+u(x)$ is a singleton, then $u$ is differentiable at $x$. Moreover, if $D^+u(x)$ is a singleton for every point in $\Omega$, then $u\in C^1(\Omega)$.
  \item $D^+u(x)=\mathrm{co}\, D^{\ast}u(x)$.
  \item $D^{\ast}u(x)=\big\{\lim_{i\to\infty}p_i~:~ p_i\in D^+u(x_i),\; x_i\to x,\;\mathrm{diam}\,(D^+u(x_i))\to 0\big\}$.
\end{enumerate}
\end{Pro}

%It is well known that for the study more general functions, e.g., local Lipschitz functions, we need the following concept of {\em Clarke gradients}. The readers can refer to the monograph \cite{Clarke} for details.
%
%\begin{defn}
%Let $u:\Omega\to\R$ be a local Lipschitz function, fix $x\in\Omega$ and $\theta\in\R^n$, the generalized lower derivative is defined as
%$$
%u^{\circ}_-(x,\theta)=\liminf_{\stackrel{\scriptstyle h\to 0^+}{y\to x}}\frac{u(y+h\theta)-u(y)}h.
%$$
%Moreover the set
%$$
%\partial_Cu(x)=\{p\in\R^n: u^{\circ}_-(x,\theta)\leqslant\langle p,\theta\rangle, \forall\theta\in\R^n\}
%$$
%is called {\em Clark's gradient} of $u$ at $x$.
%\end{defn}
%
%A similar statement as Theorem \ref{basic_facts_of_superdifferential}(d) holds for the Clarke's gradients, i.e.,
%$$
%\partial_Cu(x)=\text{co}D^{\ast}u(x)
%$$
%for any local Lipschitz function $u$. It implies that for any semiconcave function $u$, the superdifferential $D^+u(x)$ coincides with the Clarke's gradients $\partial_Cu(x)$.

A point $x\in\Omega$ is called a {\em singular point} of $u$ if $D^+u(x)$ is not a singleton. The set of all singular points of $u$, also called the {\em singular set} of $u$, is denoted by $\Sigma_u$.

\subsection{Hamilton-Jacobi equations and viscosity solutions}
Throughout this paper we will be concerned with the Hamilton-Jacobi equation
\begin{equation}\label{mech_sys_general}
H\big(x,c+Du(x)\big)=\alpha(c)\qquad (x\in\T^n)
\end{equation}
with
\begin{equation}\label{eq:Hamilton}
H\big(x,p\big):=\frac 12\big\langle A(x)p,p\big\rangle+V(x)\qquad (x,p)\in\T^n\times \R^n\,,
\end{equation}
where
\begin{itemize}
\item  $\langle A(\cdot)\cdot,\cdot\rangle$ is a smooth $\Z^n$-periodic Riemannian metric on $\R^n$ (that is,
$A(x)$ is a $\Z^n$-periodic symmetric positive definite $n\times n$ matrix with smooth entries),
\item $V$ is a smooth $\Z^n$-periodic  function on $\R^n$,
\item $c\in\R^n$, and
\item $\alpha:\R^n\to\R$ is Mather's $\alpha$-function (that is, for any $c\in\R^n$, $\alpha(c)$ is the unique constant such that \eqref{mech_sys_general} admits a  $\Z^n$-periodic viscosity solution).
\end{itemize}
We recall that a continuous function $u$ is called a {\em viscosity subsolution} of equation
\eqref{mech_sys_general} if, for any $x\in\T^n$,
\begin{align}\label{viscosity subsolution}
H(x,c+p)\leqslant\alpha(c),\quad\forall p\in D^+u(x)\,.
\end{align}
Similarly, $u$ is a {\em viscosity supersolution} of equation \eqref{mech_sys_general} if, for any $x\in\T^n$,
\begin{align}\label{viscosity supersolution}
H(x,c+p)\geqslant\alpha(c),\quad\forall p\in D^-u(x)\,.
\end{align}
Finally, $u$ is called a {\em viscosity solution} of equation \eqref{mech_sys_general}, if it is both a viscosity subsolution and a supersolution.

\begin{Rem}
In the context of weak KAM theory, the configuration space $M$ is a smooth closed manifold in general. In this case, we can use the standard definition of  viscosity solutions of the Hamilton-Jacobi equation \eqref{mech_sys_general} given in terms of  test functions, i.e., $u$ is called a viscosity subsolution (resp. supersolution) of \eqref{mech_sys_general} if
$$
H(x,c+D\phi(x))\leqslant\alpha(c),\quad(\text{resp.}\ H(x,c+D\phi(x))\geqslant\alpha(c)),
$$
for any $\phi\in C^1(\Omega)$ such that $u-\phi$ attains a local maximum (resp. minimum) at $x$, and a viscosity solution of \eqref{mech_sys_general} if it is a viscosity subsolution and a viscosity supersolution simultaneously. However, for the purposes of this paper, we only consider the case of $M=\T^n$. Thus, we also omit the general treatment of the notion of local semiconcavity on manifolds (see, for instance, \cite{Rifford}).
\end{Rem}

\subsection{Facts from nonsmooth analysis}\label{se:NSA}
We now need to recall some basic fact from nonsmooth analysis. The interested reader can find more details on this topic  in \cite{Clarke}, as well as the proof of the results we are going to recall.

Let $S\subset\R^n$ be a nonempty closed set and let $x\in S$.
\begin{defn}
A vector $\theta\in\R^n$ belongs to the {\em contingent cone} (or {\em Bouligand's tangent cone}) $T_S(x)$ iff there exist  sequences $\theta_i\in \R^n$ converging to $\theta$ and  $ t_i\in\R^+$ decreasing to $0$ such that
$$
x+t_i\theta_i\in S\,,\quad \forall i\geqslant 1\,.
$$
A vector $\theta\in\R^n$ belongs to {\em Clarke's tangent cone}   $C_S(x)$ iff, for all sequences $x_i\in S$ converging to $x$ and  $ t_i\in\R^+$ decreasing to $0$, there is a sequence  $\theta_i\in \R^n$
$$
x_i+t_i\theta_i\in S\,,\quad \forall i\geqslant 1\,.
$$
The vector space generated by $T_S(x)$ is called the {\em tangent space} to $S$ at $x$ and is denoted by $\text{Tan}(x,S)$.
\end{defn}
Note that $C_S(x)$ is always contained in $T_S(x)$.
\begin{defn}
$S$ is said to be {\em epi-Lipschitzian} near $x$ if there is an invertible linear transformation $A:\R^n\to \R^{n-1}\times\R$ such that, for some neighborhood $U$ of $x$, one has
\begin{equation*}
S\cap U=U\cap A^{-1}(\text{epi}\,f)\,,
\end{equation*}
where $f:\R^{n-1}\to\R$ is Lipschitz function and 
\begin{equation*}
\text{epi}\,f=\big\{(y,z)\in\R^{n-1}\times\R~:~f(y)\leqslant z\big\}\,.
\end{equation*}
\end{defn}
When the above holds, the part of the boundary of $S$ in $U$ is represented by the graph of a Lipschitz function and can be regarded as an $(n-1)$-dimensional Lipschitz submanifold  of $\R^n$. One can show that a closed set $S$ is epi-Lipschitzian near $x$ iff  $C_S(x)$ has nonempty interior (see \cite{Rockafellar}).
\begin{defn}
Let $F:\Omega\subset\R^n\to\R^m$ be a vector-valued function, written in terms of component functions as $F(x)=(f^1(x),\ldots,f^m(x))$, where each $f^i:\Omega\to\R$, $i=1,\ldots,m$, is locally Lipschitz near $x$. The {\em generalized Jacobian} of $F$ at $x$, denoted by $\partial F(x)$, is the compact set
$$
\partial F(x)=\text{co}\{\lim_{x_i\to x}JF(x_i): F\ \text{is differentiable at}\ x_i \}.
$$
$\partial F(x)$ is said to be of {\em maximal rank} if every matrix in $\partial F(x)$ is of maximal rank.
\end{defn}

We conclude with the implicit function theorem for Lipschitz mappings. Let $\Omega\subset \R^m\times\R^k$ be an open set, let
$
F:\Omega\to\R^k
$
and let $(y_0,z_0)\in \Omega $. The notation $\pi_z\partial F(y_0,z_0)$ stands for the set of all $k\times k$ matrices $M$ such that, for some $k\times m$ matrix $N$, the $k\times(k+m)$ matrix $(N,M)$ belongs to $\partial F(y_0,z_0)$.

\begin{Pro}\label{IFT}
Let $F: \Omega\subset\R^m\times\R^k\to\R^k$ be locally Lipschitz, let $(y_0,z_0)\in \Omega $, and
 suppose that $\pi_z\partial F(y_0,z_0)$ only contains matrices of maximal rank. Then there exists a neighborhood $U$ of $y_0$ and a Lipschitz function $z:U\to\R^k$ such that $z(y_0)=z_0$ and
$$
F(y,z(y))=0,\quad \forall\,y\in U.
$$
\end{Pro}

\section{propagation of singularities for weak KAM solutions}
Let $u_c(x)$ be a weak KAM solution of equation \eqref{mech_sys_general}, with Hamiltonian given by \eqref{eq:Hamilton}. One can lift the problem to the universal covering space $\R^n$ defining
\begin{equation}\label{v}
v(x):=u_c(x)+\langle c,x\rangle, \qquad x\in \R^n.
\end{equation}
Then, $v$ is also a locally semiconcave function. Throughout this paper we will assume the energy condition
\begin{equation}\label{energy_condition}
\alpha(c)>\max_{x\in\T^n} V(x),
\end{equation}
where the latter is the so-called Ma\~n\'e's critical value.

\subsection{Local maxima of weak KAM solutions}
Let $v$ be defined by \eqref{v}. Then, the singular sets of $v$ and $u_c$ coincide. For all $x\in \R^n$, define the level set
$$
\Lambda_x:=\{y\in\R^n: v(y)=v(x)\},
$$
and superlevel set
$$
\Lambda^+_x:=\{y\in\R^n: v(y)\geqslant v(x)\}.
$$
Let $x\in\Sigma_v$. We want to characterize the {\sl regularity condition}
$$
0\not\in D^+v(x)
$$
in terms of the smoothness of the level set $\Lambda_x$. As is well-known, the above condition is  crucial for the propagation of singularities along generalized characteristics.

%A point $x$ is said to be a {\em local minimum} (resp. {\em maximum}) point of $v$ if there exist $r>0$, such that $v(y)\geqslant v(x)$ (resp. $v(y)\leqslant v(x)$) for every $y\in B(x,r)$.

\begin{ex}
Let us consider the mechanical system $L(x,q)=\frac 12|q|^2-V(x)$, where $x\in\T^n$ and $V\leqslant0$ is such that $\max_{x\in\T^n}V(x)=0$ and $\min_{x\in\T^n}V(x)<0$. Let $\mu$ be a $c$-minimal invariant probability measure on $T\T^n$ (see, e.g. \cite{Mather91},\cite{Mather93}), that is,
$$
-\alpha(c)=\int_{T\T^n}\Big(\frac 12|q|^2-\langle c,q\rangle-V(x)\Big) d\mu\,.
$$
If $\kappa=\int_{T\T^n} V(x)d\mu_0>\min_{x\in\T^n}V(x)$, where $\mu_0$ is an invariant probability measure  generated by the linear flow $\dot{x}(t)=c$, then
$$
\frac12|c|^2+\kappa\leqslant\alpha(c)\leqslant\frac12|c|^2.
$$
Therefore,
$$
|c|\leqslant\sqrt{2(\alpha(c)-\kappa)}\,.
$$
On the other hand, if $v$ is defined as in \eqref{v}, easy calculations show that, for any $x\in\text{arg}\min V$,
$$
|p|=\sqrt{2(\alpha(c)-V(x))}\,,\quad p\in D^{\ast}v(x)\,.
$$
%This explains the possible existence of local maximum point of $v$, even $v$ is surjective.
This implies the Lipschitz constant of $v$ near $x$ is not less than $|c|$. It is clear that there could exist no local maximum point of $v$, if $|c|$ were larger than the Lipschitz constant of $v$.
\end{ex}

\begin{Pro}
Let $v$, defined  in \eqref{v}, satisfy the energy condition \eqref{energy_condition}. Then $v$ has no classical critical point.
Consequently, if $x$ is a local maximum point of $v$, then $x\in\Sigma_v$. Moreover, $v$ has no local minima.
\end{Pro}

\begin{proof}
For any $x\in\R^n$, if $v$ is differentiable at $x$ with $Dv(x)=p$, then $D^+v(x)$ is a singleton and $D^+v(x)=D^{\ast}v(x)=\{p\}$ with
$
|p|=\sqrt{2(\alpha(c)-V(x))}>0
$
in view of the energy condition \eqref{energy_condition}. Therefore, $p\neq 0$.
Since $0\in D^+v(x)$ whenever $v$ has a  local maximum at $x$, we deduce that any point at which $v$ has a local maximum must be a singular. Finally,  $D^{\pm}v(x)\neq \varnothing$  at any point $x$ at which $v$ attains a local minimum. So, any such point would be a classical critical point of $v$.
\end{proof}

The above proposition implies that  any $x\in\Sigma_v$ is either a ``saddle point''  or a local maximum point.

\begin{The}\label{equiv}
For $n\geqslant 2$, let $v$, defined as in \eqref{v}, satisfy the energy condition \eqref{energy_condition} and let $x\in\Sigma_v$. Then $0\not\in D^+v(x)$ if and only if $x$ is not a local maximum point of $v$ and $\Lambda_x$ is an $(n-1)$-dimensional Lipschitz submanifold near $x$.
\end{The}

\begin{proof}
Assume first  $0\not\in D^+u(x)$. Then $x$ cannot be a local maximum point of $v$ and, by Proposition \ref{IFT}, $\Lambda_x$ is an $(n-1)$-dimensional Lipschitz manifold near $x$.

Conversely, suppose $x$ is not a local maximum point of $v$ and $\Lambda_x$ is an $(n-1)$-dimensional Lipschitz manifold near $x$. Then $\Lambda^+_x$ is epi-Lipschitzian near $x$ and, as  recalled in section~\ref{se:NSA}, Clarke's tangent cone to $\Lambda^+_x$ at $x$ has nonempty interior. Since $C_{\Lambda^+_x}(x)\subset T_{\Lambda^+_x}(x)$,  
for any $\theta\in C_{\Lambda^+_x}(x)$ with $|\theta|=1$
there are  sequences $\theta_i\in \R^n$ converging to $\theta$ and  $ t_i\in\R^+$ decreasing to $0$ such that
$
x+t_i\theta_i\in \Lambda^+_x
$ for all $i$.
By the semiconcavity of $v$, for any $q\in D^+v(x)$ and all $i$, we have
\begin{equation}\label{basic_facts_of_semiconcavity}
v(x)\leqslant v(x+t_i\theta_i)\leqslant v(x)+\langle q,t_i\theta_i\rangle+\frac C2t_i^2|\theta_i|^2\,,
\end{equation}
which leads to
$$
\langle q,\theta_i\rangle+\frac C2t_i|\theta_i|^2\geqslant 0
$$
and, eventually, to  $\langle q,\theta\rangle\geqslant 0$ for all $q\in D^+v(x)$. Therefore, for any set  $\{\theta_k\}_{k=1}^n$ of linearly independent
 unit vectors generating $C_{\Lambda^+_x}(x)$ we have that
\begin{equation}\label{eq:intersection-half-space}
\langle q,\theta_k\rangle\geqslant 0,\quad\ \text{for all}\ q\in D^+v(x),\quad k=1\ldots,n.
\end{equation}
In other terms,  $D^+v(x)$ is contained in the convex cone
\begin{equation*}
C=\big\{q\in\R^n~:~\langle q,\theta_k\rangle\geqslant 0, \,k=1,\ldots,n \big\}\,.
\end{equation*}
Observe that $0$ is an extreme point of such a cone since $\{\theta_k\}_{k=1}^n$ are linearly independent. Were $0$  also in $D^+v(x)$,  one would have that $0\in\text{Ext}\, D^+v(x)$. This contradicts with Proposition \ref{Ext_and_reachable} which implies that
\begin{equation*}
\text{Ext}\,D^+v(x)=D^{\ast}v(x)\subset\big\{p\in\R^n~:~\langle A(x)p,p\rangle=\rho\big\}\,,
\end{equation*}
where 
$\rho=\sqrt{2(\alpha(c)-V(x))}>0$ in view of condition \eqref{energy_condition}.
\end{proof}

\begin{ex}
Theorem \ref{equiv} does hold true if $v$ is a general semiconcave function. For example, let $v:\R^2\to\R$ be defined as
$$
v(x,y)=
\begin{cases}
\hspace{.cm}
x^2&x\leqslant 0\,,\; y\in \R
\\
\hspace{.cm}
\frac 1{x+1}-1&x> 0\,,\; y\in \R\,.
\end{cases}
$$
It is easy to verify that $(0,0)$ is not a local maximum of $v$ and
\begin{equation*}
\Lambda_{(0,0)}=\{(0,y): y\in\R\}
\end{equation*}
is an $(n-1)$-dimensional Lipschitz manifold, but $D^+v((0,0))=[-1,0]\times\{0\}\ni(0,0)$.
\end{ex}

\begin{Rem}\label{homoclinic_by_sing}
%It is interesting to note the following fact from the proof of Theorem \ref{equiv}: 
Let $x\in\Sigma_{v}$ be a local maximum point of $v$,
%, when we want to prove the regularity result of Theorem \ref{equiv}, we can not exclude the case if 
and suppose there exist
% a pair of 
antipodal points $p_1, p_2\in D^{\ast}v(x)$, $p_1=-p_2$. Then, by Proposition~\ref{reachable_grad_and_backward}, there exist  $C^1$ arcs $\gamma_1$ and $\gamma_2$ defined on $(-\infty,0]$, which are $(u_c,L_c,\alpha(c))$-calibrated, and $p_i=\frac{\partial L_c}{\partial q}(x,\dot{\gamma_i}(0))$, $i=1,2$. For the mechanical system
$$
H(x,p)=\frac 12\langle A(x)p,p\rangle+V(x),
$$
we have that $\dot{\gamma}_i(0)=A(x)p_i$, $i=1,2$, and $\dot{\gamma}_1(0)=-\dot{\gamma}_2(0)$. Define
$$
\gamma(t)=\left\{
         \begin{array}{ll}
           \gamma_1(t), & \hbox{$t\leqslant0$;} \\
           \gamma_2(-t), & \hbox{$t>0$.}
         \end{array}
       \right.
$$
It is easy to see $\gamma:(-\infty,+\infty)\to\T^n$ is a $C^1$ extremal curve. Moreover, it is well known that both the $\alpha$-limit and the $\omega$-limit of $\gamma$ belong to (certain Aubry classes of) $\mathscr{A}_c$, see, e.g., \cite{Bernard2002}. In fact, $(\gamma,\dot{\gamma})$ is a homoclinic orbit with respect to $\tilde{\mathscr{A}}_c$.

The above situation can indeed occur in weak KAM theory. For example, for the one-dimensional standard mathematical pendulum, there exists a closed interval $[c^-,c^+]$ near $0$ and $\alpha(c)\equiv 0$ when $c$ is contained in such a interval. If $c^-<c<c^+$, then it is well known there is one singular point of the corresponding weak KAM solution, say $x_c$,  and such a point produces a homoclinic orbit to the hyperbolic fixed point in the aforementioned way. Note that of $x\in\Sigma_{u_c}$ for $c\in(c^-,c^+)$, then $x\not\in\mathcal{I}(u_c)$ since $u_c$ is differentiable at each point of $\mathcal{I}(u_c)$. However, $\mathcal{I}(u_c)=\T^1$ for $c=c^{\pm}$ and each $x\in\mathcal{I}(u_{c^{\pm}})\setminus\mathscr{M}_{c^{\pm}}$ also gives a homoclinic orbit in $\tilde{\mathscr{A}}_{c^{\pm}}$ respectively (for the definition of the set $\mathcal{I}(u)$ for a weak KAM solution $u$, see, e.g., \cite{Fathi-Siconolfi2004}). 
%Unfortunately, maybe we have no example in higher dimension under the energy condition \eqref{energy_condition} yet.
\end{Rem}

\subsection{Generalized characteristics}
\label{se:GC}
The concept of generalized characteristic was initiated by Dafermos \cite{Dafermos}, and systematically developed later, see \cite{Albano-Cannarsa,Cannarsa-Sinestrari,Cannarsa-Yu} and the reference therein.

\begin{defn}
A Lipschitz arc $\mathbf{x}:[0,\tau]\to\R^n$ is said to be a {\em generalized characteristic} of the Hamilton-Jacobi equation
$$
H\big(x,Du(x)\big)=E
$$
with energy $E\in\R$, if $\mathbf{x}$ satisfies the differential inclusion
\begin{equation}\label{generalized_characteristics}
\dot{\mathbf{x}}(s)\in\mathrm{co}\, H_p\big(\mathbf{x}(s),D^+u(\mathbf{x}(s))\big),\quad \text{a.e.}\;s\in[0,\tau]\,.
\end{equation}
\end{defn}

A basic criterion for the propagation of singularities along generalized characteristic was given in \cite{Albano-Cannarsa} (see \cite{Cannarsa-Yu,Yu} for an improved version and simplified proof).

\begin{Pro}[\cite{Albano-Cannarsa}]\label{criterion_on_gen_char}
Let $u$ be a viscosity solution of Hamilton-Jacobi equation
$$
H(x,Du(x))=E,
$$
with $E$ not less than Ma\~n\'e's critical value, and let $x_0\in\R^n$. Then there exists a unique  generalized characteristic  $\mathbf{x}:[0,\tau]\to\R^n$ with initial point $\mathbf{x}(0)=x_0$. Moreover, if 
$x_0\in\Sigma_u$, then $\mathbf{x}(s)\in\Sigma_u$ for all $s\in [0,\tau]$. Furthermore, if
\begin{equation}\label{condition_for_propagation_singularities}
0\not\in\mathrm{co} \,H_p(x_0,D^+u(x_0))\,,
\end{equation}
 then  $\mathbf{x}(\cdot)$ is injective for every $s\in[0,\tau]$.
\end{Pro}

Note that, for mechanical systems,  condition \eqref{condition_for_propagation_singularities} is just $0\not\in A(x)D^+v(x)$, which is equivalent to
\begin{equation}\label{reg_condition}
0\not\in D^+v(x),
\end{equation}
since $A(x$) is positive definite. The  characteristic differential inclusion is
\begin{equation*}
\dot{\mathbf{x}}(s)\in A(\mathbf{x}(s))D^+v(\mathbf{x}(s)).
\end{equation*}

The following properties of generalized characteristics are fundamental (see \cite{Cannarsa-Yu}).

\begin{Pro}\label{uniqueness_of_general_charact}
If $\mathbf{x}:[0,\tau]\to\R^n$ is a generalized characteristic, then the right derivative $\dot{\mathbf{x}}^+(s)$ exists for all $s\in[0,\tau)$ and
\begin{equation}
\label{generalized_characteristics_mech_sys}
\dot{\mathbf{x}}^+(s)=A(\mathbf{x}(s))p(s)\qquad\forall s\in[0,\tau),
\end{equation}
where $p(s)$ is the unique point of $D^+v(\mathbf{x}(s))$ such that
\begin{equation}\label{minimality}
\frac 12\langle A(\mathbf{x}(s))p(s),p(s)\rangle=\min_{q\in D^+v(\mathbf{x}(s))}\frac 12\langle A(\mathbf{x}(s))q,q\rangle.
\end{equation}
Moreover, $\dot{\mathbf{x}}^+$ is right-continuous.
\end{Pro}

We now turn to study the behavior of a weak KAM solution $v$ along a generalized characteristic.

\begin{The}\label{increasing}
Let $v$, defined as \eqref{v}, satisfy the energy condition \eqref{energy_condition}, and let $\mathbf{x}:[0,\tau]\to\R^n$ be a generalized characteristic such that $\dot{\mathbf{x}}^+(s)\not=0$, $\forall s\in [0,\tau]$. Then
\begin{equation}\label{eq:increasing}
v(\mathbf{x}(s_1))<v(\mathbf{x}(s_2)),\quad\ \text{for all}\quad 0\leqslant s_1<s_2\leqslant\tau.
\end{equation}

Moreover, if $x\in\Sigma_v$  is not a local maximum point of $v$ and the level set $\Lambda_x$ is an $(n-1)$-dimensional Lipschitz manifold near $x$, then there exists a singular generalized characteristic $\mathbf{x}:[0,\tau']\to\R^n$ with $\mathbf{x}(0)=x$ and $\dot{\mathbf{x}}^+(s)\not=0$ for all $s\in [0,\tau']$. In particular, $\mathbf{x}$ is not constant.
\end{The}

\begin{proof}
The proof of \eqref{eq:increasing} is similar to that of \cite[Theorem 1]{acns}.
Let 
\begin{equation*}
p(s):=A^{-1}(\mathbf{x}(s))\dot{\mathbf{x}}^+(s)\in D^+v(\mathbf{x}(s))\,, \qquad s\in[0,\tau)\,,
\end{equation*}
like in Proposition \ref{uniqueness_of_general_charact}. 
It suffices to show that the right derivative of $v(\mathbf{x}(s))$ is positive, since then \eqref{eq:increasing} follows by integration (recall that $v(\mathbf{x}(s))$ is locally Lipschitz).
Let $s_0\in[0,\tau)$, and set
$$
x_0=\mathbf{x}(s_0),\quad p_0=A^{-1}(x_0)\dot{\mathbf{x}}^+(s_0)\in D^+v(x_0).
$$
Recall that $v$ is a viscosity solution of
$$
\frac 12\langle A(x)Dv(x),Dv(x)\rangle+V(x)=E:=\alpha(c),
$$
and set $\bar{\rho}=\sqrt{2(E-V(x_0))}$, and denote by $E_{\rho}(x_0)$ the ellipsoid
\begin{equation}\label{eq:ellipsoid}
E_{\rho}(x_0)=\big\{p\in\R^n~:~\langle p, A(x_0)p\rangle\leqslant\rho\big\}
\end{equation}
for all $0<\rho\leqslant\bar{\rho}$. From Proposition \ref{uniqueness_of_general_charact}, it follows that
$$
\rho_0:=\langle p_0, A(x_0)p_0\rangle\leqslant\langle p, A(x_0)p\rangle,\quad \forall\,p\in D^+v(x_0)\,.
$$
Being $\langle p, A(x_0)p\rangle$ strictly convex, we have that $E_{\rho_0}(x_0)\cap D^+v(x_0)=\{p_0\}$. Since both $E_{\rho_0}(x_0)$ and $D^+v(x_0)$ are compact convex subsets of $\R^n$, the support hyperplane to $E_{\rho}(x_0)$ at $p_0$ separates $E_{\rho}(x_0)$ and $D^+v(x_0)$, i.e.
$$
\langle p_0, A(x_0)p_0\rangle\leqslant\langle p, A(x_0)p_0\rangle,\quad p\in D^+v(x_0).
$$
Then, using the Lipschitz continuity of $v$ and the representation of the directional derivative of a semiconcave function (see \cite[Theorem 3.3.6]{Cannarsa-Sinestrari}), we have
\begin{align*}
\frac d{ds^+}v(\mathbf{x}(s))_{|_{s=s_0}}&=\lim_{h\to0^+}\frac{v(x_0+h\dot{\mathbf{x}}^+(s_0))-v(x_0)}h\\
                             &=\min_{p\in D^+v(x_0)}\langle p,\dot{\mathbf{x}}^+(s_0)\rangle\\
                             &=\min_{p\in D^+v(x_0)}\langle p,A(x_0)p_0\rangle\\
                             &=\langle p_0,A(x_0)p_0\rangle.
\end{align*}
Since $p(s)\not=0$ for every $ s\in [0,\tau]$ by  assumption, \eqref{eq:increasing} is proved.

In order to prove the last part of the conclusion, observe that the regularity condition $0\not\in D^+v(x)$ holds true by Theorem \ref{equiv}. Then, the existence of the singular generalized characteristic is obtained by Proposition \ref{criterion_on_gen_char}. Using the upper semicontinuity of $D^+v$,  for  $\varepsilon>0$ small enough, there exists $\delta>0$ such that $D^+v(\mathbf{x}(s))$ is contained in the $\varepsilon$-neighborhood of $D^+v(x)$ for all $s\in[0,\delta]$. Therefore, $0\not\in D^+v(\mathbf{x}(s))$, or, by \eqref{minimality}, $\dot{\mathbf{x}}^+(s)\not=0$ for all $s\in[0,\delta]$.
\end{proof}

\begin{Cor}
For any $x\in\Sigma_v$, $0\not\in D^+v(x)$ if and only if there exists a singular generalized characteristic with initial point $x$ such that $\dot{\mathbf{x}}^+(0)\not=0$.
\end{Cor}

\begin{proof}
By Proposition \ref{criterion_on_gen_char}, for any $x\in\Sigma_v$ with $0\not\in D^+v(x)$, there is a singular generalized characteristic $\mathbf{x}:[0,\tau]\to\R^n$ with initial point $x$. Moreover, $\dot{\mathbf{x}}(s)\not=0$ for all $s\in [0,\tau]$ by Theorem \ref{increasing}. 

Conversely, if $\mathbf{x}:[0,\tau]\to\R^n$ is a singular generalized characteristic with $\mathbf{x}(0)=x$ and $\dot{\mathbf{x}}^+(0)\not=0$, then $0\not\in D^+v(x)$ by  \eqref{generalized_characteristics_mech_sys} and \eqref{minimality}.
\end{proof}

\subsection{Local propagation of singularities}
Let $v$ be defined as in \eqref{v}. Recall that $v$ is a locally semiconcave function with linear modulus.

Let us denote by $S^{n-1}$ the
$(n-1)$-dimensional sphere $\{x\in\R^n: |x|=1\}$. The following proposition, that we state without proof, is the extension of a result due to \cite{Cannarsa-Yu}, in the case of the flat metric on $\T^n$, to a general Riemannian metric.

\begin{Pro}\label{pr:wKAM}
Let $u_c$ be a $\T^n$-periodic locally semiconcave solution of \eqref{mech_sys_general} satisfying the energy condition \eqref{energy_condition}. Let $x_0\in\Sigma_{u_c}$ and let $\Omega$ be an arbitrary bounded open set containing $x_0$ such that $\partial \Omega$ is homeomorphic to $S^{n-1}$. Then $\partial \Omega\cap \Sigma_{u_c}\ne\varnothing$.
\end{Pro}

We will now show that singularities actually propagate along a Lipschitz continuous curve, answering the question raised in \cite[Remark~6.2]{Cannarsa-Yu}.

\begin{The}\label{th:wKAM}
Let $u_c$ be a $\T^n$-periodic semiconcave solution of \eqref{mech_sys_general} satisfying the energy condition \eqref{energy_condition}. If  $x_0\in\Sigma_{u_c}$, then there exists a Lipschitz arc
$\mathbf{x}:[0,\tau)\to\R^n$ such that $\mathbf{x}(0)=x_0\,,\;\dot{\mathbf{x}}^+(0)\neq 0$, and $\mathbf{x}(t)\in\Sigma_{u_c}$ for all $t\in[0,\tau)$.
\end{The}
\begin{proof}
Let $u_c$ be a viscosity solution of \eqref{mech_sys_general} and $v$ be defined as \eqref{v}. Observe that $v$ is also locally semiconcave and has the same singular set as $u_c$. In view of the propagation principle of \cite{alca99}---which applies to any semiconcave function---the geometric condition
\begin{equation}\label{eq:c1}
\partial D^+v(x_0)\setminus D^{\ast}v(x_0)\neq\varnothing
\end{equation}
implies the conclusion of the theorem. Moreover, since
 \begin{equation*}
D^*v(x_0)\subset \partial E_\rho(x_0)
\end{equation*}
with $E_{\rho}(x_0)$ defined in \eqref{eq:ellipsoid} for $\rho=\sqrt{2(\alpha(c)-V(x_0))}>0$, we have that \eqref{eq:c1} is equivalent to
\begin{equation}\label{eq:c2}
D^*v(x_0)\neq \partial E_\rho(x_0)\,
\end{equation}
by Proposition \ref{Ext_and_reachable}. We shall therefore prove \eqref{eq:c2} .

Suppose $D^{\ast}v(x_0)= \partial E_\rho(x_0)$ and, for each $\theta\in\partial E_\rho(x_0)$, denote by $(\mathbf{x}_\theta,p_\theta)$ the solution on $(-\infty,0]$ of the Hamiltonian system
\begin{equation}\label{eq:hamiltonian}
\begin{cases}
\Dot{\mathbf{x}}=A(\mathbf{x})p\\
\Dot p=\frac 12D\langle A(\mathbf{x})p,p\rangle-DV(\mathbf{x})
\end{cases}
\quad\text{with initial conditions}\quad
\begin{cases}
\mathbf{x}(0)=x_0
\\
p(0)=\theta\,.
\end{cases}
\end{equation}
It is well-known  that $v$ is differentiable in a neighborhood of $\mathbf{x}_\theta(t)$ for all $t< 0$ and
\begin{equation}\label{eq:cpKAM}
Dv(\mathbf{x}_\theta (t))=p_\theta(t)\quad \forall t\in(-\infty,0)\,.
\end{equation}
Moreover, $\theta\mapsto (\mathbf{x}_\theta(t),p_\theta(t))$ is a continuous map for every $t\le 0$, and so is the map
\begin{equation*}
\gamma_t:\partial E_\rho(x_0)\to \R^n\,,\quad \gamma_t(\theta)=\mathbf{x}_\theta(t)\,.
\end{equation*}
We claim that $\gamma_t$ is one-to-one for every $t<0$. For let $\theta_0, \theta_1\in \partial B_\rho$ be such that $\gamma_t(\theta_0)=\gamma_t(\theta_1)$. Then, by the forward uniqueness of solutions to the differential inclusion $\Dot{\mathbf{x}}\in A(\mathbf{x})D^+v(\mathbf{x})$ (Proposition \ref{uniqueness_of_general_charact}), we conclude that $\mathbf{x}_{\theta_0}(s)=\mathbf{x}_{\theta_1}(s)$ for all $s\in [t,0]$. Thus, in view of \eqref{eq:cpKAM}, $p_{\theta_0}(s)=p_{\theta_1}(s)$ for all $s\in [t,0]$. So, $\theta_0=\theta_1$.

Since $\gamma_t$ is continuous and one-to-one, we have that 
$\Gamma_t:=\gamma_t(\partial E_\rho)$ is homeomorphic to $S^{n-1}$  for all $t<0$. So,  the Jordan-Brouwer theorem (see, e.g., \cite[Corollary~18.7]{gr67}) ensures that, for any $t< 0$, $\R^n\setminus \Gamma_t$ has two connected components both with boundary $\Gamma_t$, one of which---labeled $\Omega_t$---is bounded. Moreover, $v$ is of class $C^1$ in a neighborhood of  $\Gamma_t$. Since, by Proposition~\ref{pr:wKAM}, $\Omega_t\cap\Sigma_v=\varnothing$, we conclude that  $v$ is of class $C^1$ on an open neighborhood, say $U$, of $\overline{\Omega}_t$. In fact, a well-know regularization argument based on semiconcavity and semiconvexity shows that $v\in C^{1,1}_{\rm loc}(U)$.

Now, let $x_1\in\Omega_t$ and let $\mathbf{x}_1:[0,+\infty)\to\R^n$ be the solution of the problem
\begin{equation}\label{eq:xi1}
\begin{cases}
\Dot{\mathbf{x}}(s)\in A(\mathbf{x}(s))D^+u(\mathbf{x}(s))
&s\ge 0
\\
\mathbf{x}(0)=x_1.
\end{cases}
\end{equation}
Define
\begin{equation*}
\tau:=\inf\{s\ge 0~|~\mathbf{x}_1(s)\in \Gamma_t\}>0\,.
\end{equation*}
Since $\mathbf{x}_1(s)\in \Omega_t$ for all $s\in [0,\tau)$, we have that
 \begin{multline*}
 %\label{eq:boundonxi}
\langle c,\mathbf{x}_{1}(s)-x_1\rangle+v(\mathbf{x}_{1} (s))-v(x_1)=\int_{0}^{s}\langle Du(\mathbf{x}_{1} (\sigma)), \dot{\mathbf{x}}(\sigma)\rangle\,d\sigma
 \\
=2\int_{0}^{s}\left[\alpha(c)-V(\mathbf{x}_{\theta} (\sigma))\right]\,d\sigma
\geqslant 2\mu s\,,\quad\forall s\in [0,\tau)\,,
 \end{multline*}
with $\mu=E-\max_{\T^n}V>0$. Therefore,
 \begin{equation}
 \label{eq:Pshift}
\langle c, \mathbf{x}_{1} (s)-x_1\rangle\geqslant 2 (\mu s -\|v\|_\infty)\,,\quad\forall s\in [0,\tau)\,.
\end{equation}
The above inequality  forces $\tau<+\infty$. So, $\mathbf{x}_1(\tau)\in \Gamma_t$. Equivalently, $\mathbf{x}_1(\tau)=\mathbf{x}_\theta(t)$ for some $\theta\in\partial E_\rho(x_0)$. Then, since \eqref{eq:xi1} admits a unique solution, we have that
\begin{equation}\label{eq:xi1=xi}
\mathbf{x}_1(s)=\mathbf{x}_\theta(t+s-\tau)
\end{equation}
for any $s\ge \tau$. Moreover, since $u\in C^{1,1}_{\rm loc}(U)$, we also have that \eqref{eq:xi1=xi} holds for all $s\ge 0$. In particular, $\mathbf{x}_\theta(t-\tau)=x_1$.

Finally, the same reasoning that led to \eqref{eq:Pshift} ensures that
 \begin{equation*}
\langle c, \mathbf{x}_{\theta}(t')\rangle\leqslant \langle c, x_1\rangle+2\|v\|_\infty - 2\mu (t-\tau-t')\,,
 \end{equation*}
 for every $t'<t-\tau$. Hence, for some $T<t-\tau$, $\mathbf{x}_{\theta} (T)$ must return to intersect  $\Gamma_t$, that is,
\begin{equation}\label{eq:contra}
\mathbf{x}_\theta(T)=\mathbf{x}_{\theta'}(t)\quad \text{for some}\quad \theta'\in\partial E_\rho(x_0)\,.
\end{equation}
We claim this  yields a contradiction. Indeed, \eqref{eq:contra} and again forward uniqueness  imply that
\begin{equation*}
\mathbf{x}_\theta(T+s-t)=\mathbf{x}_{\theta'}(s)\quad \text{for all}\quad t\le s\le 0\,.
\end{equation*}
 Hence, for $s=0$, we obtain  $\mathbf{x}_\theta(T-t)=x_0$.  Now, this is a contradiction, because $v$ is differentiable at $\mathbf{x}_\theta(T-t)$ whereas $x_0\in\Sigma_v$. We have reached such a contradiction supposing that $D^*v(x_0)= \partial E_\rho(x_0)$, so \eqref{eq:c2} must hold true.
\end{proof}

\begin{Lem}\label{not_constant}
Let $\mathbf{x}:[a,b]\to\R^n$ be a Lipschitz arc, and let $v$ be a locally semiconcave function. If $v\big(\mathbf{x}([a,b])\big)$ contains a nontrivial closed interval, then there exists $s_0\in [a,b]$ such that $0\not\in D^+v(\mathbf{x}(s_0))$.
\end{Lem}

\begin{proof}
Suppose  $0\in D^+v(\mathbf{x}(s))$ for all $s\in [a,b]$. Then, by the semiconcavity of $v$, 
$$
v(\mathbf{x}(s))-v(\mathbf{x}(s'))\leqslant\frac C2|\mathbf{x}(s')-\mathbf{x}(s)|^2\leqslant K|s-s'|^2\quad \forall\, s,s'\in [a,b],
$$
where $K\geqslant 0$ depends on the Lipschitz norm of  $\mathbf{x}$. It follows that the derivative of $v(\mathbf{x}(s))$ vanishes for all $s\in (a,b)$. This implies that $v(\mathbf{x}(s))$ is constant on $[a,b]$ since $v(\mathbf{x}(s))$ is Lipschitz continuous, which leads to a contradiction.
\end{proof}

\begin{Cor}\label{local_regular}
Let $v$, defined as in \eqref{v}, satisfy the energy condition \eqref{energy_condition}, and let $x\in\Sigma_v$ be an isolated local maximum point of $v$. Then, for any neighborhood $U$ of $x$, there exists $y\in\Sigma_v\cap U$ such that $0\not\in D^+v(y)$.
\end{Cor}

\begin{proof}
Let $x\in\Sigma_v$ be an isolated local maximum point of $v$. Then, by Theorem \ref{th:wKAM},  there exists a Lipschitz singular arc $\mathbf{x}:[0,\tau)\to \R^n$ such that $\mathbf{x}(0)=x$ and $\dot{\mathbf{x}}^+(0)\neq 0$. Therefore, 
in any neighborhood $U$ of $x$  one can find a point $\mathbf{x}(s)\in\Sigma_v$ with $0<s<\tau$ so that $v\big(\mathbf{x}(s)\big)<v(x)$. The conclusion follows by Lemma~\ref{not_constant}.
\end{proof}

\begin{Rem}
By Corollary \ref{local_regular}, if $x\in\Sigma_v$ is an isolated local maximum point of $v$, then  in any neighborhood $U$ of $x$ there is a point $y\in\Sigma_v$ satisfying the regularity property $0\not\in D^+v(y)$. Whether the generalized characteristic $\mathbf{x}(s)$ with  initial point $\mathbf{x}(0)=y$, given by Proposition \ref{criterion_on_gen_char},  reaches $x$ in finite or infinite time remains, however, an open problem.
\end{Rem}

\section{Applications to weak KAM theory}
In this section, we will use standard notions from weak KAM theory. The reader will find all  the necessary preliminaries  in Appendix A.

Let $c\in\R^n$ and let $h_c$ be Peierls' barrier (see Definition~\ref{de:peierls} below). The {\em barrier function} $B_c^{\ast}(x)$ is defined as
\begin{equation}\label{defn_barrier}
B_c^{\ast}(x)=\inf_{y,z\in\mathscr{M}_c}\{h_c(y,x)+h_c(x,z)-h_c(y,z)\},\quad x\in\T^n,
\end{equation}
where $\mathscr{M}_c$ is the projected Mather set, that is, the projection on $\T^n$ of the Mather set $\tilde{\mathscr{M}}_c$ by the graph property. Note that $\mathscr{M}_c\subset\mathscr{A}_c$  (see, e.g., \cite{Bernard2002}\cite{Mather91}\cite{Mather93}).
By Proposition~\ref{h_c_determin_viscosity_solution}, $h_c(x,\cdot)$ defines a global viscosity subsolution of \eqref{mech_sys_general}. Also,  $h_c(\cdot,x)$ defines a global critical supersolution. Fix $y,z\in\mathscr{M}_c$ and, for each $x\in\T^n$, let
$$
u_{c,y}^-(x)=h_c(y,x),\quad u_{c,z}^+(x)=-h_c(x,z).
$$
Then
\begin{equation}\label{B_c_star}
B^{\ast}_c(x)=\inf_{y,z\in\mathscr{M}_c}\{u_{c,y}^-(x)-u_{c,z}^+(x)-h_c(y,z)\}.
\end{equation}

For any $x,y\in\T^n$ define  {\em Mather's pseudometric} (see \cite{Mather93}) on $\mathscr{A}_c$ by
\begin{equation*}
d_c(x,y)=h_c(x,y)+h_c(y,x)\,.
\end{equation*}
Two points  $x,y\in\mathscr{A}_c$ are said to be in the same {\em Aubry class} if $d_c(x,y)=0$.

\begin{Lem}\label{uniqueness}
Let $x,y\in\mathscr{A}_c$ be distinct points in the same Aubry class. Then $h_c(x,\cdot)$ equals $h_c(y,\cdot)$ up to a constant. If $x,y\in\mathscr{A}_c$, $x\ne y$, belong to different Aubry classes, then $h_c(x,\cdot)-h_c(y,\cdot)$ is not constant.
\end{Lem}

\begin{proof}
First of all, let us recall  that
\begin{equation}\label{eq_9}
h_c(x,z)=h_c(x,y)+h_c(y,z),
\end{equation}
if either $d_c(x,y)=0$ or $d_c(y,z)=0$ (see \cite[p. 1370]{Mather93}). Let us define $u_c^x(\cdot)=h_c(x,\cdot)$ for $x\in\T^n$. Then $u_c^x$ equals $u_c^y$ up to a constant.

On the other hand, fix $x,y\in\T^n$ belonging to distinct Aubry classes, and set $f(z)=h_c(x,z)-h_c(y,z)$, $z\in\T^n$, then
\begin{align*}
f(x)-f(y)&=h_c(x,x)-h_c(y,x)-h_c(x,y)+h_c(y,y)\\
         &=-\big(h_c(y,x)+h_c(x,y)\big)=-d_c(x,y)\not=0\,.
\end{align*}
Thus $h_c(x,\cdot)-h_c(y,\cdot)\not\equiv C$ for any constant $C$.
\end{proof}

From \eqref{eq_9} it follows  that, if there exists a unique Aubry class, then we can represent the barrier function $B^{\ast}_c$ in the form
$$
B^{\ast}_c(x)=u_{c,y}^-(x)-u_{c,y}^+(x):=u_c^-(x)-u_c^+(x)\,,
$$
where $(u_c^-,u_c^+)$  is called a {\em conjugate pair} of weak KAM solutions.

A set $\mathcal{S}$ of Tonelli Lagrangians is said to be {\em generic} in the sense of Ma\~n\'e if there exists a residual set $\mathcal{O}\subset C^2(\T^n)$, and a Tonelli Lagrangian $L_0$, such that  each $L\in\mathcal{S}$ has the form
\begin{equation*}
L=L_0+V
\end{equation*}
 for some $V\in\mathcal{O}$. A similar notion can be given for a set of Tonelli Hamiltonians. Moreover, a well-known result by 
 Ma\~n\'e~\cite{Mane} ensures that, for any fixed $c\in\R^n$, there is a unique viscosity solution of the equation associated with a generic Hamiltonian.

It is well known that $u_c^-$ (resp. $u_c^+$) is a locally semiconcave (resp. semiconvex) function with linear modulus. Then the barrier function $B^{\ast}_c$ is also a locally seminconcave function with linear modulus, see, e.g., \cite[Proposition 2.1.5]{Cannarsa-Sinestrari}.

The following theorem describes the propagation of singularities of the barrier function.
\begin{The}\label{pro_barrier}
For any generic mechanical system as \eqref{mech_sys_general} with
$$
L(x,v)=L_0(x,v)-V(x)=\frac 12 \langle A^{-1}(x)v,v\rangle-V(x),
$$the singularities of the barrier function $B^{\ast}_c$ always propagate locally if the energy condition \eqref{energy_condition} is satisfied.
\end{The}

\begin{proof}
Recall that, in generic case,
$$
B^{\ast}_c(x)=u_c^-(x)-u_c^+(x).
$$
Let $x$ be a singular point of $B^{\ast}_c$. Then, we have that
\begin{equation}\label{superdiff_of_barrier}
x\in\Sigma_{u_c^-}\cup\Sigma_{u_c^+}\quad\text{and}\quad D^+B^{\ast}_c(x)=D^+u^-_c(x)-D^-u^+_c(x)\,.
\end{equation}
Indeed, if $x\not\in\Sigma_{u_c^-}\cup\Sigma_{u_c^+}$, then both $u_c^-$ and $u_c^+$ are differentiable at $x$, so $x$ is not a singular point of $B^{\ast}_c$. Moreover, in view of \eqref{eq:SCC2}, the representation of  $D^+B^{\ast}_c(x)$ in \eqref{superdiff_of_barrier} follows by the sum rule for the superdifferential of concave functions.

If the energy condition \eqref{energy_condition} is satisfied, then we have
$$
\partial D^+u^-_c(x)\setminus D^{\ast}u^-_c(x)\not=\varnothing\quad\text{or}\quad \partial D^-u^+_c(x)\setminus D^{\ast}u^+_c(x)\not=\varnothing
$$
by the same argument we used in the proof of Theorem \ref{th:wKAM}. This implies that the singularity of $u^-_c$ or $u^+_c$ must propagate locally. Our conclusion follows by \eqref{superdiff_of_barrier}.
\end{proof}

The following result is motivated by Remark \ref{homoclinic_by_sing}.

\begin{The}\label{homoclinic_orbit}
Let $x\in\Sigma_{B^{\ast}_c}$, and let $B^{\ast}_c(x)=u^-_c(x)-u^+_c(x)$ where $(u^-_c,u^+_c)$ is a conjugate pair of weak KAM solutions in the generic case. Then $x$ produces a homoclinic orbit with respect to the Aubry set $\tilde{\mathscr{A}}_c$ if
\begin{equation}\label{nonempty}
D^{\ast}{u^-_c}(x)\cap D^{\ast}{u^+_c}(x)\not=\varnothing.
\end{equation}
If $x$ is a local minimum point of $B^{\ast}_c$, then $x$  also produces such a homoclinic orbit.
\end{The}

\begin{proof}
Let $x\in\Sigma_{B^{\ast}_c}$, and let $p\in\ D^{\ast}{u^-_c}(x)\cap D^{\ast}{u^+_c}(x)$. Then, $|c+p|\not=0$ since $c+p\in\partial E_{\rho}(x)$ with $\rho=\sqrt{2(\alpha(c)-V(x))}>0$, where $E_{\rho}(x)$ is the ellipsoid defined in \eqref{eq:ellipsoid}. Since such a vector $p$ is a reachable gradient of $u^{\pm}_c(x)$, it follows that there exists a $C^1$ arc $\gamma_1:(-\infty,0]\to\R^n$ (resp. $\gamma_2:[0,+\infty)\to\R^n$), with $\gamma_1(0)=x$ (resp. $\gamma_2(0)=x$), such that $\gamma_1$ is backward $(u^-_c,L_c,\alpha(c))$-calibrated (resp. $\gamma_2$ is forward $(u^+_c,L_c,\alpha(c))$-calibrated) by Proposition \ref{reachable_grad_and_backward}. Moreover, 
\begin{equation*}
c+p=\frac{\partial L}{\partial q}(x,\dot{\gamma}_i(0))=A(x)\dot{\gamma}_i(0)\,,\quad i=1,2\,.
\end{equation*}
Now, define
$$
\gamma(t)=\left\{
         \begin{array}{ll}
           \gamma_1(t), & \hbox{$t\leqslant0$;} \\
           \gamma_2(t), & \hbox{$t>0$.}
         \end{array}
       \right.
$$
Then, $(\gamma,\dot{\gamma})$ is a $C^1$ extremal arc, and produces a homoclinic orbit with respect to the Aubry set $\tilde{\mathscr{A}}_c$.

Finally, if $x$ is a local minimum point of $B^{\ast}_c$, then $B^{\ast}_c$ is differentiable at $x$ with $D^+B^{\ast}_c(x)=\{0\}$. Thus,  the sum rule for the superdifferential of concave functions yields that both $u^-_c$ and $u^+_c$ are differentiable at $x$, and $Du^-_c(x)=Du^+_c(x)$. So, condition \eqref{nonempty} is satisfied and,  by the first part of the conclusion, $x$ produces a homoclinic orbit with respect to the Aubry set.
\end{proof}

\begin{Rem}
The assumptions of Theorem \ref{homoclinic_orbit} are satisfied by the mathematical pendulum system, as discussed in Remark~\ref{homoclinic_by_sing} where condition \eqref{nonempty} always holds for $c\in (c^-,c^+)$. It is interesting to compare this analysis with the technique that uses the set $\mathcal{I}(u^-,u^+)$ in classical weak KAM theory (see, e.g., \cite{Fathi-book}). On the other hand, how to guarantee that condition \eqref{nonempty} holds remains  an open problem.
\end{Rem}

It is well known that complex  phenomena of Hamiltonian dynamics occur when the unstable and stable manifolds of hyperbolic periodic orbits intersect transversally. From the variational viewpoint, this is closely related to the regularity  of the barrier function, as well as the structure of its singular set and conjugate loci. 
%We will continue our study of the fine properties of singularities with applications to Hamiltonian dynamics in the future.

\appendix
\section{Semiconcavity and weak KAM theory}
We being this appendix by briefly surveying some basic notions from weak KAM theory. We refer the reader to Fathi's unpublished book \cite{Fathi-book} for more details.

\subsection{A brief review of weak KAM solutions}
%Let $c=[\eta_c]\in H^1(M,\R)$, and $L$ be a $C^2$ Tonelli Lagrangian. Define $L_c=L-\eta_c$, for any closed 1-form $\eta_c$ in the cohomology class $c$, which is well defined since the action $\int\eta\ d\mu=0$ for any invariant measure with respect to the Euler-Lagrange flow if $\eta$ is an exact 1-form, and the Euler-Lagrange equations with respect to $L$ and $L_c$ is the same since $\eta_c$ is closed.

%Now we concentrate on the case when $M=\T^n$. Since $H^1(\T^n,\R)\simeq\R^n$, then we can choose a vector $c\in\R^n$ as a representative element in the corresponding cohomological class $c$. Here we abuse $c$ for both the cohomological class and the representative element. 
Let $H$ be a Tonelli Hamiltonian on $n$-torus $\T^n$, and $L$ be the corresponding Tonelli Lagrangian. For any fixed vector $c\in\R^n$ let us consider the Hamilton-Jacobi equation
\begin{equation}\label{HJE_weak_KAM}
H\big(x,c+Du_c(x)\big)=E:=\alpha(c)\,,\quad x\in\T^n\,,
\end{equation}
where the energy $E$ is assumed to satisfy the energy condition \eqref{energy_condition}.

\begin{defn}
Let $L$ be a $C^2$ Tonelli Lagrangian on $T\T^n$, and set $L_c=L-c$ and $E=\alpha(c)$ for any $c\in\R^n$. A function $u_c:\T^n\to\R$ is said to be {\em dominated} by $L_c+E$ iff for each absolutely continuous arc $\gamma:[a,b]\to\T^n$ with $a<b$, one has
$$
u_c(\gamma(b))-u_c(\gamma(a))\leqslant\int^b_aL_c(\gamma(s),\dot{\gamma}(s))ds+E(b-a).
$$
When this happens, one writes $u_c\prec L_c+E$.
\end{defn}

\begin{defn}
Let $c\in\R^n$, and $u_c$ be a real-valued function on $\T^n$. A absolutely continuous curve $\gamma:[a,b]\to\T^n$ is said to be $(u_c,L_c,E)$-{\em calibrated} if
$$
u_c(\gamma(b))-u_c(\gamma(a))=\int^b_aL_c(\gamma(s),\dot{\gamma}(s))ds+E(b-a).
$$
\end{defn}

The following facts clarifying the relation of the viscosity solutions and weak KAM solutions are well known (see, e.g. \cite{Fathi-book} \cite{Fathi-Siconolfi2004}).

\begin{Pro}
Let $c\in\R^n$. A function $u_c:\T^n\to\R$ is dominated by $L_c+E$ if and only if $u_c$ is a viscosity subsolution of \eqref{HJE_weak_KAM}. 

If $u_c$ is a viscosity subsolution of \eqref{HJE_weak_KAM}, then there exists an absolutely continuous arc $\gamma_x:(-\infty,0]\to \T^n$ with $\gamma_x(0)=x$ such that $\gamma_x$ is $(u_c,L_c,E)$-calibrated.
\end{Pro}
A viscosity subsolution  (resp. solution) of \eqref{HJE_weak_KAM} is also called a {\em critical} subsolution (resp. solution).

For $t>0$, $x,y\in\T^n$ and $c\in\R^n$, we introduce the following
quantity
\begin{equation}
h^c_t(x,y)=\inf\int^t_0L_c(\gamma(s),\dot{\gamma}(s))\ ds,
\end{equation}
where $\gamma$ belongs to the set $C_{x,y}(t)$ of all absolutely continuous arcs $\gamma:[0,t]\to\T^n$ such that
$\gamma(0)=x$ and $\gamma(t)=y$.

\begin{defn}\label{de:peierls} Let $c\in\R^n$ and let $h^c_t(x,y)$ be defined as above. {\em Ma\~n\'e's critical potential}
and {\em Peierls' barrier} are defined, respectively, as
\begin{align}
\phi_c(x,y)&=\inf_{t>0}h^c_t(x,y)+\alpha(c)t\,,\\
h_c(x,y)&=\liminf_{t\to\infty}h^c_t(x,y)+\alpha(c)t\,.
\end{align}
We call $\mathscr{A}_c=\big\{x\in\T^n~:~h_c(x,x)=0\big\}$ the {\em projected Aubry set}.
\end{defn}

It is well known that $\mathscr{A}_c$ is nonempty for any $c\in\R^n$. 

\begin{Pro}\label{h_c_determin_viscosity_solution}
{\em (\cite{Fathi-Siconolfi2004})} If  Peierls' barrier $h_c$ is finite then, for each
  $x\in\T^n$, $u_c(y):=h_c(x,y)$ is a global critical 
  solution of \eqref{HJE_weak_KAM}.
  Moreover, for any $x,y\in\T^n$, there is an arc $\xi:(-\infty,0]\to\T^n$, with
  $\xi(0)=y$, such that
  $$
  u_c(\xi(0))-u_c(\xi(-t))=\int^0_{-t}L_c(\xi(s),\dot{\xi}(s))\ dt+\alpha(c)t,\quad \forall\,t\geqslant0\,.
  $$
\end{Pro}

\begin{Pro}
{\em (\cite{Fathi-Siconolfi2004})} For each
  $x\in\T^n$, $u_c(y):=\phi_c(x,y)$ is a global critical 
  subsolution of \eqref{HJE_weak_KAM}. Moreover, $u_c$ defined as above is a global critical 
  solution  if and only if $x\in\mathscr{A}_c$.
  Furthermore, for any $y\in\T^n\setminus\{x\}$, there is an arc $\xi:(-\varepsilon,0]$, with
  $\xi(0)=y$, such that
  $$
  u_c(\xi(0))-u_c(\xi(-t))=\int^0_{-t}L_c(\xi(s),\dot{\xi}(s))\ dt+\alpha(c)t,\quad \forall\,t\in[0,\varepsilon]\,.
  $$
\end{Pro}

\begin{Pro}\label{subsolution_admit_solution}
{\em (\cite{Fathi-book})} Let $w$ be a critical  subsolution of \eqref{HJE_weak_KAM}. Then there exists be a critical  solution $u$ such that $u_{|\mathscr{A}_c}=w_{|\mathscr{A}_c}$.
\end{Pro}

Since $H$ is convex in the fibers, one has that the set of all critical subsolutions is  convex , that is, for any  pair of critical  subsolution $v_0,v_1$ and every $\lambda\in (0,1)$, $v_{\lambda}=\lambda v_0+(1-\lambda)v_1$ is also a critical  subsolution. So, if \eqref{HJE_weak_KAM} admits two distinct critical  solutions, then each $v_{\lambda}$ admits a critical  solution $u_{\lambda}$, and \eqref{HJE_weak_KAM} admits infinitely many critical  solutions by Proposition \ref{subsolution_admit_solution}.

\subsection{Viscosity solutions and their semiconcavity}
Now, we recall some properties related to the semiconcavity of  viscosity solutions. The following result is fundamental (see, e.g., \cite{Fathi-book} \cite{Rifford}).

\begin{Pro}
Any viscosity solution of the Hamilton-Jacobi equation \eqref{HJE_weak_KAM} is locally semiconcave with linear modulus.
\end{Pro}
The following is the weak KAM analogue of \cite[Theorem~6.4.12]{Cannarsa-Sinestrari}
\begin{Pro}\label{Ext_and_reachable}
$\mathrm{Ext}\,D^+u(x)=D^{\ast}u(x)$ for any viscosity solution $u$ of \eqref{HJE_weak_KAM} and any $x\in\T^n$.
\end{Pro}

\begin{proof}
The inclusion $\mathrm{Ext}\,D^+u(x)\subset D^{\ast}u(x)$ is a direct consequence of Proposition~\ref{basic_facts_of_superdifferential}~(d). For the opposite inclusion, fix $x\in\T^n$ and let $p$ be a reachable gradient vector of $u$ at $x$. Then there exists a sequence $\{x_k\}$ such that $u$ is differentiable at each $x_k$, $H\big(x_k,c+Du(x_k)\big)=E$ and $p=\lim_{k\to\infty}Du(x_k)$. Therefore,
\begin{align*}\label{reachable_gradient}
H(x,c+p)=E,\quad\forall p\in D^{\ast}u(x)\,.
\end{align*}
Then the strict convexity of $H$ in the fibers implies that $p$ is no convex combination of other points of $D^+u(x)$.  Thus, $D^{\ast}u(x)\subset \mathrm{Ext}\,D^+u(x)$.
\end{proof}

\begin{Rem}
Note that the equality $\mathrm{Ext}\,D^+u(x)=D^{\ast}u(x)$ is false for a general semiconcave function $u$,  see e.g. \cite[Example~3.3.13]{Cannarsa-Sinestrari}. 
\end{Rem}

 We now turn to discuss  some connections between the dynamics of Hamiltonian flows on an energy hypersurface with $E$ not less than Ma\~n\'e's critical value $c_0$, and the structure of the superdifferential of the viscosity solutions of \eqref{HJE_weak_KAM}. The main part of the following result is due to Rifford \cite{Rifford}, see Lemma 6 and Lemma 7 therein.

\begin{Pro}\label{reachable_grad_and_backward}
Let $x\in \T^n$ and $u:\T^n\to\R$ be a viscosity solution of the Hamilton-Jacobi equation
$$
H(x,c+Du(x))=E=\alpha(c),\quad x\in \T^n,
$$
where $E\geqslant c_0$. Then $p\in D^{\ast}u(x)$ if and only if there exists a unique $C^1$ curve $\gamma:(-\infty,0]\to \T^n$ with $\gamma(0)=x$ which is $(u,L,E)$-calibrated, and $p=\frac{\partial L}{\partial q}(x,\dot{\gamma}(0))$.
\end{Pro}

\begin{proof}
Let $x\in \T^n$ and $p\in D^{\ast}u(x)$. Then there exists a sequence $\{x_k\}$, with $\lim_{k\to\infty}x_k=x$, such $u$ is differentiable at each $x_k$ with $p_k=Du(x_k)\to p$. It is well known from weak KAM theory that, for each $k$, there exists a unique $C^1$ arc $\gamma_k:(-\infty,0]\to \T^n$ which is $(u,L,E)$-calibrated, and $p_k=\frac{\partial L}{\partial q}(x,\dot{\gamma_k}(0))$. The sequence $\gamma_k$ is equi-Lipschitz, so, by the Ascoli-Arzela theorem, we can extract  a subsequence converging to a $C^1$ arc $\gamma$ which is $(u,L,E)$-calibrated, and $p=\frac{\partial L}{\partial q}(x,\dot{\gamma}(0))$. Such an arc is necessarily a solution of the Euler-Lagrange equations. Uniqueness follows from  classical results for ordinary differential equations.

Conversely, suppose there is a unique $C^1$ arc $\gamma:(-\infty,0]\to \T^n$ with $\gamma(0)=x$ which is $(u,L,E)$-calibrated. Then, taking any sequence $x_k=\gamma(t_k)$ such that $t_k<0$ and $x_k\to x$ as $k\to\infty$, one has that $u$ is differentiable at $x_k$ and 
\begin{equation*}
Du(x_k)=\frac{\partial L}{\partial q}(x_k,\dot{\gamma}(t_k))\to \frac{\partial L}{\partial q}(x,\dot{\gamma}(0))=p\,.
\end{equation*}
Thus, $p\in D^{\ast}u(x)$.
\end{proof}

\end{document}